\documentclass[graybox]{svmult}

\usepackage{mathptmx}       
\usepackage{helvet}         
\usepackage{courier}        
\usepackage{type1cm}        
\usepackage{makeidx}         
\usepackage{graphicx}        
\usepackage{multicol}        
\usepackage[bottom]{footmisc}

\usepackage{amsmath, amssymb,bbm,enumerate}
\usepackage{mathrsfs}

\makeindex             

\newcommand{\eps}{\varepsilon}
\newcommand{\weak}{\rightharpoonup}
\newcommand{\one}{\mathbbm{1}}
\newcommand{\CR}{\mathbb{R}}

\newcommand{\CN}{\mathbb{N}}

\newcommand{\loc}{\mathrm{loc}}
\newcommand{\A}{\mathscr{A}}

\newcommand{\cL}{\mathscr{L}}
\newcommand{\half}{\frac{1}{2}}
\newcommand{\ip}[2]{\ifthenelse{\equal{#1}{}}{\mbox{$ [ \,\cdot\, , \, \cdot \, ] $}}{
\mbox{$ \left[ #1 \, , \, #2 \right]$}}}
\newcommand{\norm}[1]{\ifthenelse{\equal{#1}{}}{\mbox{$\|\cdot\|$}}{\mbox{$\| #1 \|$}}}
\newcommand{\dual}[2]{\ifthenelse{\equal{#1}{}}{\mbox{$ \langle \,\cdot\; , \; \cdot \, \rangle $}}{
\mbox{$ \langle #1   ,  #2 \rangle$}}}

\newtheorem{hyp}[theorem]{Hypothesis}

\allowdisplaybreaks

\begin{document}

\title*{Kernel estimates for nonautonomous Kolmogorov equations with potential term}
\titlerunning{Kernel estimates for nonautonomous Kolmogorov equations with potential term}
\author{Markus Kunze, Luca Lorenzi\thanks{Thank you dad, for having conveyed your love for mathematics to me!}, Abdelaziz Rhandi\\
\vspace{1cm}
\hspace*{2cm}{\sc To the memory of Prof. Alfredo Lorenzi}}
\authorrunning{M. Kunze, L. Lorenzi, A. Rhandi}
\institute{Markus Kunze \at Graduiertenkolleg 1100, University of Ulm, 89069 Ulm, Germany, \email{markus.kunze@uni-ulm.de}
\and Luca Lorenzi \at Dipartimento di Matematica e Informatica, Universit\`a degli Studi di Parma, Parco Area delle Scienze 53/A, 43124 Parma, Italy, \email{luca.lorenzi@unipr.it}
\and Abdelaziz Rhandi \at Dipartimento di Ingegneria dell'Informazione, Ingegneria Elettrica e Matematica Applicata, Universit\`a degli Studi di Salerno, Via Ponte Don Melillo 1, 84084 Fisciano (Sa), Italy,
\email{arhandi@unisa.it}}
%
%

\maketitle

\abstract{Using time dependent Lyapunov functions, we prove pointwise upper bounds for the heat kernels of some  nonautonomous Kolmogorov operators with possibly unbounded drift and diffusion coefficients and a possibly unbounded potential term.}

\section{introduction}
We consider nonautonomous evolution equations
\begin{equation}\label{eq.nee}
\left\{ \begin{array}{rlll}
\partial_t u(t,x) & = & \A (t)u(t,x), & (t,x) \in (s, 1]\times \CR^d,\\[1mm]
u(s,x) & = & f(x), & x \in \CR^d\, ,
\end{array}\right.
\end{equation}
where the time dependent operators $\A (t)$ are defined on smooth functions $\varphi$ by
\[
\A (t)\varphi (x) = \sum_{ij=1}^dq_{ij}(t,x)D_{ij}\varphi (x) + \sum_{i=1}^dF_i(t,x)D_i\varphi (x) - V(t,x)\varphi (x).
\]
We write $\A_0(t)$ for the operator $\A(t) +V(t)$.
Throughout this article, we will always assume that the following hypothesis on the coefficients are satisfied.
\begin{hyp}\label{hyp1}
The coefficients $q_{ij}, F_j$ and $V$ are defined on $[0,1]\times \CR^d$ for $i,j=1, \ldots d$. Moreover,
\begin{enumerate}
\item there exists an $\varsigma \in (0,1)$ such that $q_{ij}, F_j, V \in C^{\frac{\varsigma}{2}, \varsigma}_\loc ([0,1]\times \CR^d)$ for all $i,j =1, \ldots, d$. Further, $q_{ij} \in C^{0,1}((0,1)\times \CR^d)$;
\item
the matrix $Q = (q_{ij})$ is symmetric and uniformly elliptic in the sense that there exists a number $\eta > 0$ such that
\begin{eqnarray*}
\sum_{i,j=1}^d q_{ij}(t,x)\xi_i\xi_j \geq \eta |\xi|^2 \quad \mbox{for all}\,\, \xi \in \CR^d,\,\, (t,x) \in [0,1]\times \CR^d;
\end{eqnarray*}
\item $V \geq 0$;
\item
there exist a nonnegative function $Z\in C^2(\CR^d)$ and a constant $M \geq 0$ such that $\lim_{|x|\to\infty}Z(x) = \infty$
and we have $\A(t)Z(x) \leq M$, as well as $\eta \Delta_x Z(x) + F(t,x)\cdot \nabla_xZ(x)-V(t,x)Z(x) \leq M$, for all $(t,x) \in [0,1]\times \CR^d$;
\item
there exists a nonnegative function $Z_0\in C^2(\CR^d)$ such that $\lim_{|x|\to\infty}Z_0(x) = \infty$
and we have $\A_0(t)Z_0(x) \leq M$, as well as $\eta \Delta_x Z_0(x) + F(t,x)\cdot \nabla_x Z_0(x) \leq M$, for all $(t,x) \in [0,1]\times \CR^d$.
\end{enumerate}
\end{hyp}

We summarize Hypothesis \ref{hyp1}(4)-(5) saying that $Z$ (resp. $Z_0$) is {\it a Lyapunov function} for the operators $\A$ and $\eta\Delta +F\cdot \nabla_x -V$
(resp. for the operators $\A_0$ and $\eta\Delta +F\cdot\nabla_x$).

Clearly, $5$ implies $4$. However, for applications it will be important to differentiate between $Z$ and $Z_0$.

The previous assumptions guarantee that, for any $f\in C_b(\CR^d)$, the Cauchy problem \eqref{eq.nee} admits a unique solution $u \in C_b([s,1]\times \CR^d) \cap C^{1,2}((s,1]\times \CR^d)$. Moreover,
there exists an evolution family $(G(t,s))_{(t,s) \in D} \subset \cL (C_b(\CR^d))$, where $D = \{(t,s) \in [0,1]^2 : t \geq s\}$, which governs Equation \eqref{eq.nee}, i.e.,
 $u(t,x) = (G(t,s)f)(x)$. Here and throughout the paper, the index ``$b$'' stands for boundedness.

By \cite[Proposition 3.1]{al11}, the operators $G(t,s)$ are given by \emph{Green kernels} $g(t,s, \cdot, \cdot)$, i.e.,\ we have
\begin{equation}
G(t,s)f(x) = \int_{\CR^d} f(y)g(t,s, x,y)\, dy.
\label{kernel-G}
\end{equation}

Our aim is to prove estimates for the Green kernel $g$. Similar results as we present here have been obtained in \cite{mpr10, ms07, MST, s08} for autonomous equations without potential term. The case of autonomous equations with potential term was treated in \textcolor{red}{\cite{alr10,lmpr11,lr12}}. Recently, generalizing techniques from \cite{ffmp09} to the parabolic situation, the
authors of the present article extended these results also to nonautonomous equations and, even more importantly, allowed also unbounded diffusion coefficients, see \cite{klr13}. In this article, we extend the results of \cite{klr13} to also allow potential terms in the equation.

Applying our main abstract result (Theorem \ref{t.main}) in a concrete situation, we obtain the following result.
In its formulation,
for $s \geq 0$, we use the notation $|x|_{*}^s$ to denote a smooth version of the $s$-th power of the absolute value function, i.e., $|x|_*^s = |x|^s$ whenever $|x|\geq 1$ and the map $x \mapsto |x|_*^s$ is twice continuously differentiable in $\CR^d$.
This is done to meet the differentiability requirement in Hypothesis \ref{hyp1}(1), 3 and 5 and also later differentiability requirements.
If $s=0$ or $s >1$ we can choose
$|x|_*^s = |x|^s$ for any $x\in\CR^d$ as this is already twice continuously differentiable.

\begin{theorem}\label{t.example}
Let $k>d+2$, $m, r\geq 0$ and $p> 1$ be given with $p> m-1$ and $r > m-2$. We consider the (time independent) operator $\A (t) \equiv \A$, defined on smooth functions $\varphi$ by
\[
\A \varphi (x) = (1+|x|_{*}^m)\Delta\varphi (x) - |x|^{p-1}x\cdot \nabla \varphi (x) - |x|^r\varphi(x).
\]
Then we have the following estimates for the associated Green kernel $g$:
\begin{enumerate}
\item if $p\geq \half (m+r)$, then for $\alpha > \frac{p+1-m}{p-1}$ and $\eps < \frac{1}{p+1-m}$ we have
\[
g(t,s,x,y) \leq C(t-s)^{1-\frac{\alpha(m\vee p)k}{p+1-m}} e^{-\eps (t-s)^\alpha |y|_{*}^{p+1-m}};
\]
\item if $p < \half (m+r)$, then for $\eps < \frac{2}{r+2-m}$ and $\alpha > \frac{r-m+2}{r+m-2}$, if $r+m>2$, and
$\alpha>\frac{r+2-m}{2(p-1)}$, if $r+m\le 2$, we have
\[
g(t,s,x,y) \leq C(t-s)^{1-\frac{\alpha(2m\vee 2p\vee r)k}{r+2-m}} e^{-\eps (t-s)^\alpha |y|_{*}^{\frac{1}{2}(r+2-m)}},
\]
for all $x,y \in \CR^d$ and $s \in [0,t)$.
\end{enumerate}
Here, $C$ is a positive constant.
\end{theorem}

These bounds should be compared to the ones in \cite[Example 3.3]{alr10}, where the case $m=0$ was considered. We would like to note that in Theorem \ref{t.example} we have restricted ourselves to the autonomous situation so that one can compare the results with those in \cite{alr10}. Genuinely nonautonomous examples can easily be constructed along the lines of \cite[Section 5]{klr13}.

\section{Time dependent Lyapunov functions}
In this section we introduce time dependent Lyapunov functions and prove that they are integrable with respect to the measures
$g_{t,s}(x, dy):=g(t,s,x,y)dy$, where
$g(t,s,\cdot,\cdot)$ is the Green kernel associated to the evolution operator $G(t,s)$, see \eqref{kernel-G},
and $g(t,\cdot,x,\cdot)\in L^1((0,1)\times\CR^d)$.
To do so, it is important to have information about the derivative of $G(t,s)f$ with respect to $s$.
We have the following result, taken from \cite[Lemma 3.4]{al11}. Here and in the rest of the paper, the index ``$c$''
stands for compactly supported.

\begin{lemma}\label{l.sderivative}
\begin{enumerate}
\item
For $f \in C_c^2(\CR^d)$, $s_0\leq s_1 \leq t$ and $x \in \CR^d$ we have
\begin{equation}\label{eq.sderivative1}
G(t,s_1)f(x) - G(t,s_0)f(x) = -\int_{s_0}^{s_1}G(t,\sigma)\A (\sigma )f(x)\, d\sigma.
\end{equation}
\item For $f \in C^2(\CR^d)$, constant and positive outside a compact set, the function $G(t,\cdot )\A (\cdot)f(x)$
is integrable in $[0,t]$ and for $s_0\leq s_1 \leq t$ we have
\begin{eqnarray*}
G(t,s_1)f(x) - G(t,s_0)f(x) \geq - \int_{s_0}^{s_1}G(t,\sigma)\A(\sigma)f(x)\, d\sigma.
\end{eqnarray*}
\end{enumerate}
\end{lemma}

We note that in the case where $V\equiv 0$ part (2) in Lemma \ref{l.sderivative} follows trivially from part (1), since
in that situation $G(t,s)\one \equiv \one$ and $\A(t)\one = 0$ so that equation \eqref{eq.sderivative1} holds
for $f=\one$, cf.\ \cite[Lemma 3.2]{kll10}.

Let us note some consequences of Lemma \ref{l.sderivative} for later use. First of all, part (1) of the lemma implies that
$\partial_sG(t,s)f = -G(t,s)\A (s)f$ for $f \in C^2_c(\CR^d)$. Arguing as in \cite[Lemma 2.2]{klr13}, we see
that for $0\leq a \leq b \leq t$, $x \in \CR^d$ and  $\varphi \in C^{1,2}_c([a,b]\times\CR^d)$, the function $s\mapsto G(t,s)\varphi (s)(x)$ is differentiable in $[a,b]$ and
\[
\partial_sG(t,s)\varphi (s)(x) = G(t,s)\partial_s\varphi (s)(x) - G(t,s)\A (s)\varphi (s)(x).
\]
Consequently, for such a function $\varphi$ we have that
\begin{equation}\label{eq.weak}
\int_a^b G(t,s)\big[\partial_s \varphi (s) -\A (s)\varphi (s)\big](x)\, ds = G(t,b)\varphi (b)(x) - G(t,a)\varphi (a)(x),
\end{equation}
for every $x \in \CR^d$.\medskip

As a consequence of formula \eqref{eq.weak} and \cite[Corollary 3.11]{bkr06} we get the following result.
\begin{lemma}
\label{lem-reg-g}
For any $t\in (0,1]$ and any $x\in\CR^d$ the function $g(t,\cdot,x,\cdot)$ is continuous (actually, locally H\"older continuous) in
$(0,t)\times\CR^d$.
\end{lemma}

We now introduce time dependent Lyapunov functions.

\begin{definition}
\label{d.lyap}
Let $t \in (0,1]$. A \emph{time dependent Lyapunov function $($on $[0,t])$} is a function
$0 \leq W \in C([0,t]\times\CR^d)\cap C^{1,2}((0,t)\times\CR^d)$ such that
\begin{enumerate}
\item
$W(s,x) \leq Z(x)$ for all $(s,x) \in [0,t]\times \CR^d$;
\item
$\lim_{|x|\to \infty}W(s,x) = \infty$, uniformly for $s$ in compact subsets of $[0,t)$;
\item there exists a function $0\leq h \in L^1((0,t))$ such that
\begin{equation}
\partial_s W(s,x) - \A(s)W(s) \geq -h(s)W(s)
\label{star}
\end{equation}
and
\begin{equation}
\partial_sW(s) - (\eta \Delta W(s)+F(s)\cdot\nabla_x W(s)-V(s)W(s)) \geq - h(s)W(s)\,,
\label{star-star}
\end{equation}
on $\CR^d$, for every $s \in (0,t)$.
\end{enumerate}
Sometimes, we will say that $W$ is a time dependent Lyapunov function \emph{with respect to $h$} to emphasize the dependence on $h$.
\end{definition}

\begin{proposition}\label{p.lyapunov}
Let $W$ be a time dependent Lyapunov function on $[0,t]$ with respect to $h$. Then for $0\leq s \leq t$ and $x \in \CR^d$
the function $W(s)$ is integrable with respect to the measure $g_{t,s}(x,dy)$. Moreover, setting
\[
\zeta_W(s,x) := \int_{\CR^d} W(s,y)g_{t,s}(x, dy)
\]
we have
\begin{equation}\label{eq.zetaest}
\zeta_W(s,x) \leq e^{\int_s^th(\tau)\, d\tau} W(t,x).
\end{equation}
\end{proposition}

\begin{proof}
Let us first note that by \cite[Proposition 4.7]{al11} the function $Z$ is integrable with respect to $g_{t,s}(x,dy)$.
Moreover,
\begin{equation}\label{eq.zest}
G(t,s)Z(x) := \int_{\CR^d} Z(y)g_{t,s}(x, dy)\leq Z(x) + M(t-s).
\end{equation}
It thus follows immediately from domination that $W(s)$ is integrable with respect to $g_{t,s}(x,dy)$.\medskip

We now fix a sequence of functions $\psi_n \in C^\infty([0,\infty))$ such that
\begin{enumerate}[(i)]
\item $\psi_n (\tau ) = \tau$ for $\tau\in [0,n]$;
\item $\psi_n(\tau) \equiv \mathrm{const.}$ for $\tau \geq n+1$;
\item $0\leq \psi_n'\leq 1$ and $\psi_n''\leq 0$.
\end{enumerate}
Let us also fix $0\le s<r<t$. Note that, for any $n\in\CN$, the function
$W_n := \psi_n\circ W$ is the sum of a function in $C_c^{1,2}([0,r]\times\CR^d)$ and a positive constant. Indeed, $W(s,\sigma) \to \infty$ as $|x| \to \infty$ uniformly on $[0,r]$.
For a positive constant function, we have by Lemma \ref{l.sderivative}(2) that
\[
G(t,r)\one - G(t,s)\one \geq - \int_s^rG(t,\sigma ) \A(\sigma)\one\, d\sigma =
\int_s^r G(t,\sigma)\big[\partial_s\one -\A(\sigma)\one\big]\, d\sigma.
\]
Combining this with Equation \eqref{eq.weak}, it follows that
\begin{align}
 &\quad G(t,r)W_n(r)(x) - G(t,s)W_n(s)(x)\notag\\
   \geq & \,\int_s^r G(t,\sigma)\big[\partial_{\sigma} W_n(\sigma) -\A(\sigma )W_n (\sigma)\big](x)\, d\sigma\notag\\
   = & \,\int_s^r G(t,\sigma ) \big[\psi_n'(W(\sigma))\big(\partial_\sigma W(\sigma) - \A(\sigma )W(\sigma)\big)\big](x)\, d\sigma\notag\\
   & \quad - \int_s^r G(t,\sigma)\big[ V(\sigma)W(\sigma)\psi_n'(W(\sigma))-V(\sigma)\psi_n(W(\sigma))\big](x)\, d\sigma\notag\\
   &\quad - \int_s^r G(t,\sigma) \big[ \psi_n''(W(\sigma))\big(Q(\sigma )\nabla_x W(\sigma)\cdot\nabla_x W(\sigma)\big)\big](x)\, d\sigma\notag\\
   \geq & \,-\int_s^r G(t,\sigma)\big[\psi_n'(W(\sigma))h(\sigma) W(\sigma)\big](x)\, d\sigma\label{eq.wnest},
\end{align}
for any $x\in\CR^d$,
since $G(t,s)$ preserves positivity and the condition $\psi_n''\le 0$ implies that $y\psi_n'(y)-\psi_n(y)\le 0$ for any $y\ge 0$.

We next want to let $r \uparrow t$. We fix an increasing sequence $(r_k)\subset (s,t)$, converging to $t$
as $k\to\infty$. By monotone convergence, we clearly have
\[
\int_s^{r_k}G(t,\sigma)\big[h(\sigma)W_n(\sigma)\big](x)\, d\sigma \to \int_s^{t}G(t,\sigma)\big[h(\sigma)W_n(\sigma)\big](x)\, d\sigma
\]
as $k\to\infty$. We now claim that $G(t,r_k)W_n(r_k)(x) \to G(t,t)W_n(t)(x) = W_n(t,x)$ as $k\to \infty$. To see this,
we note that for $f \in C_b(\CR^d)$, the function $(s,x) \mapsto G(t,s)f(x)$ is continuous in $[0,t]\times \CR^d$
as a consequence of \cite[Theorem 4.11]{al11}. This immediately implies that
$G(t,r_k)W_n(t)(x) \to G(t,t)W_n(t)(x) = W_n(t,x)$ as $k\to \infty$.
Moreover, from \eqref{eq.zest} it follows that
\begin{equation}
g_{t,s}(\CR^d\setminus B(0,R))\le \frac{1}{\inf_{\CR^d\setminus B(0,R)}Z}\int_{\CR^d}Z(y)g_{t,s}(x,dy)
\le \frac{Z(x)+M}{\inf_{\CR^d\setminus B(0,R)}Z},
\label{tight}
\end{equation}
where $B(0,R)\subset\CR^d$ denotes the open ball centered at $0$ with radius $R$,
and the right-hand side of \eqref{tight} converges to zero as $R\to\infty$. Hence,
the set of measures $\{g_{t,s}(x,dy) : s \in [0,t]\}$ is tight.

Taking into account that $W_n(r_k)$ is uniformly bounded and converges locally uniformly to $W_n(t)$ as $k\to \infty$, it is easy to see that
\[
G(t,r_k)W_n(r_k)(x) - G(t, r_k)W_n(t)(x) = \int_{\CR^d} (W_n(r_k,y) - W_n(t,y)) \, g_{t,r_k}(x,dy) \to 0
\]
as $k\to \infty$. Combining these two facts, it follows that $G(t,r_k)W_n(r_k)(x) \to W_n(t,x)$ as claimed.

Thus, letting $r\uparrow t$ in \eqref{eq.wnest}, we find that
\begin{equation}\label{eq.wnestt}
W_n(t,x) - G(t,s)W_n(s)(x) \geq - \int_s^t G(t,\sigma)\big[\psi_n'(W(\sigma)) h(\sigma) W(\sigma)\big](x)\, d\sigma.
\end{equation}
Note that $\psi_n'(W(\sigma)) h(\sigma) W(\sigma)$ and $W_n(s)$ converge increasingly to $W(\sigma)h(\sigma)$ and $W(s)$, respectively, as $n\to \infty$, for any $\sigma\in [s,t]$.
Moreover, \eqref{eq.zest} implies that $G(t,\sigma)W(\sigma)\in (0,\infty)$. Since each operator $G(t,\sigma)$ preserves positivity, we can use monotone convergence to let $n\to \infty$ in \eqref{eq.wnestt}, obtaining
\[
W(t,x) - G(t,s)W(s)(x) \geq - \int_s^th(\sigma)G(t,\sigma)W(\sigma)(x)\, d\sigma.
\]
Equivalently,
\begin{equation}\label{eq2-7}
\zeta_W(t,x) - \zeta_W(s,x) \geq -\int_s^th(\sigma)\zeta_W(\sigma,x)\, d\sigma,\qquad\;\,x\in\CR^d.
\end{equation}
This inequality yields \eqref{eq.zetaest}. Indeed, the function $\Phi$, defined by
\begin{eqnarray*}
\Phi (\tau) := \bigg( \zeta_W(t,x) + \int_\tau^t h(\sigma) \zeta_W (\sigma,x)\, d\sigma \bigg) e^{\int_s^\tau h(\sigma)\, d\sigma}\,,
\end{eqnarray*}
is continuous on $[s,t]$ and increasing since its weak derivative is nonnegative by \eqref{eq2-7}. Hence $\Phi(s)\le\Phi(t)$, from which \eqref{eq.zetaest} follows at once if we take again \eqref{eq2-7}
into account.
\qed
\end{proof}

Let us illustrate this in the situation of Theorem \ref{t.example}.

\begin{proposition}\label{p.example}
Consider the (time independent) operator $\A(t) \equiv \A$, defined by
\[
\A\varphi (x) = (1+|x|_{*}^m)\Delta\varphi (x) - |x|^{p-1}x\cdot \nabla \varphi (x) - |x|^{r}\varphi(x),
\]
where $m, r \geq 0$ and $p> 1$. Moreover, assume one of the following situations:
\begin{enumerate}[(i)]
\item $p>m-1$, $\beta := p+1-m$ and $\delta < 1/\beta$;
\item $r > m-2$, $\beta := \half (r+2-m)$ and $\delta < 1/\beta$.
\end{enumerate}
Then the following properties hold true:
\begin{enumerate}
\item the function $Z(x) := \exp (\delta |x|_{*}^\beta)$ satisfies Part (4) of Hypothesis \ref{hyp1};
\item for $0<\eps < \delta$ and $\alpha>\alpha_0$,
the function $W(s,x) := \exp (\eps (t-s)^\alpha |x|_{*}^\beta)$ is
a time dependent Lyapunov function in the sense of Definition \ref{d.lyap}. Here,
$\alpha_0 = \frac{\beta}{p-1}$ if we assume condition (ii) and additionally $m+r \leq 2$. In all other cases,
$\alpha_0=\frac{\beta}{m+\beta-2}$.
\end{enumerate}
\end{proposition}

\begin{proof}
In the computations below, we assume that $|x| \geq 1$ so that $|x|_*^s = |x|^s$ for $s\ge 0$. At the cost of
slightly larger constants, these estimates can be extended to all of $\CR^d$. We omit the details which can be obtained
as in the proof of \cite[Lemma 5.2]{klr13}

(1) By direct computations, we see that
\[
\A Z(x) = \delta \beta \Big[ (1+|x|^m)|x|^{\beta-2} \big(d+ \beta-2 + \delta \beta |x|^\beta\big)
- |x|^{p-1+\beta}-|x|^r\Big]Z(x).
\]
The highest power of $|x|$ appearing in the first term is $|x|^{m+2\beta - 2}$ which, in case (i) is
exactly $|x|^{p-1+\beta}$, in case (ii) it is exactly $|x|^r$. In both cases, the highest power in the square brackets has a negative coefficient in front, namely $\delta\beta -1$.
Thus $\lim_{|x|\to \infty}\A Z(x) = - \infty$. It now follows from the continuity
of $\A Z$ that $\A Z \leq M$ for a suitable constant $M$. Since $\eta \Delta Z+F\cdot\nabla Z-VZ\le\A Z$,
we conclude that the function $\eta\Delta Z+F\cdot\nabla Z-VZ$ is bounded from above as well.

(2) We note that since $\eps < \delta$, we have $W(s,x) \leq (Z(x))^\frac{\eps}{\delta} \leq Z(x)$ for all $s \in [0,t]$
and $x \in \CR^d$ so that (1) in Definition \ref{d.lyap} is satisfied. Condition (2) is immediate from the definition of $W$ so
that it only remains to verify condition (3).

A computation shows that
\begin{align}
&\quad \partial_sW(s,x) - \A W(s,x)\notag\\
  = &\, -\eps \alpha (t-s)^{\alpha-1}|x|^\beta W(s,x) - \eps \beta (t-s)^\alpha W(s,x)\times\label{final-est-1}\\
& \quad \times \Big[
(1+|x|^m)|x|^{\beta-2}\big( d + \beta-2 + \eps \beta (t-s)^\alpha |x|^\beta\big) - |x|^{p-1+\beta}\Big ] + |x|^rW(s,x)\notag\\
 \geq &  \,-\eps \alpha (t-s)^{\alpha-1}|x|^\beta W(s,x) - \eps \beta (t-s)^\alpha W(s,x)\times\notag\\
& \; \times \Big[
(1+|x|^m)|x|^{\beta-2}\big( d +\beta-2 + \delta \beta  |x|^\beta\big) - |x|^{p-1+\beta}\Big ] + |x|^rW(s,x)\notag\\
& \; + \eps\beta^2(\delta-\eps) (t-s)^\alpha (1+|x|^m)|x|^{2\beta-2}W(s,x)\notag\\
\geq & \,\eps (t-s)^{\alpha-1}|x|^\beta\big( (\delta -\eps) \beta^2(t-s) |x|^{m+\beta-2} - \alpha \big) W(s,x)\notag\\
&\; -\eps \beta (t-s)^\alpha W(s,x)\Big[ (1+|x|^m)|x|^{\beta-2}\big(d+\beta-2+\delta\beta |x|^{\beta}\big) -|x|^{p-1+\beta}-|x|^r\Big ],
 \label{final-est}
\end{align}
where in the last inequality we took into account that $\varepsilon\beta(t-s)^{\alpha}<1$.\medskip

To further estimate $\partial_sW(s) - \A W(s)$, we first assume that $\beta+m-2\ge 0$. This condition is satisfied under condition (i) and also under condition (ii) provided that $m+r> 2$.
We set $C:= \big[(\delta-\eps)\beta^2/\alpha\big]^{-\frac{1}{\beta+m-2}}$ and distinguish two cases.
\smallskip

{\it Case 1:} $|x| \geq C(t-s)^{-\frac{1}{\beta+m-2}}$.

In this case $(\delta -\eps) \beta^2(t-s)|x|^{\beta+m-2} \geq \alpha$ so that the first summand in \eqref{final-est} is
nonnegative. Replacing $C$ with a larger constant if necessary, we can -- as in the proof of part (1) -- ensure that
also the second summand is positive so that overall $\partial_sW(s) -\A W(s) \geq 0$ in this case.\smallskip

{\it Case 2:} $1\leq |x| < C(t-s)^{-\frac{1}{\beta+m-2}}$.

In this case, we start again from Estimate \eqref{final-est-1}. We drop the terms involving $-|x|^{p-1+\beta}$ and $|x|^r$
and, using that $|x|\geq 1$, estimate further as follows:
\begin{align*}
&\quad W(s,x)^{-1}(\partial_sW(s,x) -\A W(s,x))\\
 \geq& \,  -\eps\alpha (t-s)^{\alpha-1}|x|^\beta - 2\eps\beta(t-s)^\alpha
|x|^{m+\beta-2}(d+\beta-2 + \eps\beta |x|^\beta)\\
\geq &\, - \eps\alpha (t-s)^{\alpha-1}C^\beta (t-s)^{-\frac{\beta}{m+\beta-2}}
-2\eps\beta (t-s)^{\alpha}C^{m+\beta-2}(t-s)^{-1} \times\\
&\quad\quad\times \big(d+\beta-2 + \eps\beta C^\beta
(t-s)^{-\frac{\beta}{m+\beta-2}}\big)\\
\geq&\, -\tilde{C}(t-s)^{\alpha -1 - \frac{\beta}{m+\beta-2}}=: -h(s).
\end{align*}
Note that $h\in L^1(0,t)$ since $\alpha-1 -\frac{\beta}{m+\beta-2} > -1$ by assumption.\smallskip

Suppose now that $m+\beta-2\le 0$, so that $|x|^{m+\beta-2} \leq 1$ for $|x|\geq 1$. Taking again into account that $\eps\beta (t-s)^\alpha <1$ and dropping the term involving $|x|^r$, we derive from \eqref{final-est-1} that
\begin{align*}
&\quad W(s,x)^{-1}(\partial_sW(s,x) - \A W(s,x))\notag\\
  \geq &\, -\eps (t-s)^{\alpha-1}|x|^\beta\big (\alpha  + 2\beta-  \beta (t-s)|x|^{p-1}\big )
  -2(d+\beta-2),
\end{align*}
for any $|x|\ge 1$.
We can now argue as above taking $C=\big [(\alpha+2\beta)/\beta\big ]^{\frac{1}{p-1}}$
and distinguishing the cases $|x|\ge C(t-s)^{-\frac{1}{p-1}}$ and $1\le |x|< C(t-s)^{-\frac{1}{p-1}}$.
We conclude that
\[
W(s,x)^{-1}(\partial_sW(s,x) - \A W(s,x))\ge -\varepsilon C^{\beta}(\alpha+2\beta)(t-s)^{\alpha-1-\frac{\beta}{p-1}}
=:-h(s),
\]
for any $s\in (0,t)$, $|x|\ge 1$, and $h\in L^1((0,t))$ due to the condition on $\alpha$.

We have thus proved \eqref{star} in Definition \ref{d.lyap}. The analogous estimate \eqref{star-star}
for $\eta \Delta_{x} + F\cdot \nabla_{x} - c$ follows from observing that
$\eta \Delta_xW + F\cdot \nabla_x W - cW \leq \A W$.
\qed
\end{proof}

\section{Kernel bounds in the case of bounded diffusion coefficients}
\label{sect-3}
Throughout this section, we set $Q(a,b):=(a,b)\times\CR^d$ and $\overline{Q}(a,b):=[a,b]\times\CR^d$
for any $0\le a<b\le 1$.
Moreover, we assume that the coefficients $q_{ij}$ and their spatial derivatives $D_kq_{ij}$ are bounded on $Q(0,b)$
for $i,j,k=1, \ldots, d$ and every $b<1$.
We will remove this additional boundedness assumption in the next section.

Fix now $t \in [0,1]$. For $0\leq a < b \leq t$, $x \in \CR^d$ and $k\geq 1$, we define the quantities
$\Gamma_j(k,x,a,b)$ for $j=1,2$ by
\[
\Gamma_1(k,x,a,b) := \bigg(\int_{Q(a,b)} |F(s,y)|^k g(t,s,x,y)\, ds\, dy \bigg)^\frac{1}{k},
\]
where $g$ is the Green kernel associated with $\A$,
and
\[
\Gamma_2(k,x,a,b) := \bigg(\int_{Q(a,b)} |V(s,y)|^k g(t,s,x,y)\, ds\, dy \bigg)^\frac{1}{k}.
\]

We also make an additional assumption about the parabolic equation governed by the operators $\A_0$ without potential term.
Hypothesis \ref{hyp1}(5) guarantees that the Cauchy problem \eqref{eq.nee} with $\A$ being replaced by $\A_0$
admits a unique solution $u\in C_b(\overline Q(s,1))\cap C^{1,2}(Q(s,1))$ for any $f\in C_b(\CR^d)$. The associated evolution operator
admits a Green kernel which we denote by $g_0$. In the following lemma, we will deal with the space
$\mathscr{H}^{p,1}(Q(a,b))$ of all functions in $W^{0,1}_p(Q(a,b))$ with distributional time derivative in $(W^{0,1}_{p'}(Q(a,b)))'$,
where $1/p+1/p'=1$. We refer the reader to \cite{krylov01,mpr10} for more details on these spaces. Here, we just prove the following result which is crucial in the proof of Theorem \ref{t.mainbdd} (cf. \cite[Lemma 7.2]{mpr10}).

\begin{lemma}
\label{lem-cortona-0}
Let $u\in {\mathscr H}^{p,1}(Q(a,b))\cap C_b(\overline Q(a,b))$ for some $p\in (1,\infty)$. Then, there exists a sequence $(u_n)\subset C^{\infty}_c(\CR^{d+1})$ of smooth functions such that $u_n$ tends to $u$ in $W^{0,1}_p(Q(a,b))$ and locally uniformly in $\overline Q(a,b)$, and $\partial_tu_n$ converges to $\partial_tu$ weakly$^*$ in $(W^{0,1}_{p'}(Q(a,b)))'$ as $n\to\infty$.
\end{lemma}

\begin{proof}
We split the proof in two steps: first we prove the statement with $Q(a,b)$ being replaced with $\CR^{d+1}$ and, then, using this
result we complete the proof.

{\em Step 1.} Let $\vartheta\in C^{\infty}_c(\CR)$ be a smooth function such that $\vartheta\equiv 1$ in $(-1,1)$ and $\vartheta \equiv 0$ in $\CR \setminus(-2,2)$.
For any $\sigma>0$, any $t\in\CR$ and any $x\in\CR^d$, set $\vartheta_{\sigma}(t,x)=\vartheta(|t|/\sigma)\vartheta(|x|/\sigma)$.
Next, we define the function $u_n\in C^{\infty}_c(\CR^{d+1})$ by setting
\begin{align*}
u_n(t,x)&=n^{d+1}\vartheta_n(t,x)\int_{\CR^{d+1}}u(s,y)\vartheta_{1/n}(t-s,x-y)\,ds\,dy\\
&=:n^{d+1}\vartheta_n(t,x)(u\star \vartheta_{1/n})(t,x)\,,
\end{align*}
for any $(t,x)\in\CR^{d+1}$ and any $n\in\CN$.
Clearly, $u_n$ converges to $u$ in $W^{0,1}_p(\CR^{d+1})$ and locally
uniformly in $\CR^{d+1}$.

Let us fix a function $\psi\in W^{0,1}_{p'}(\CR^{d+1})$. Applying the Fubini-Tonelli theorem and taking into
account that $\vartheta_{1/n}(r,z)=\vartheta_{1/n}(-r,-z)$ for any $(r,z)\in\CR^{d+1}$, we easily deduce that
$\langle \partial_tu_n ,\psi \rangle=\langle \partial_tu ,\psi_n\rangle $ for any $n\in\CN$,
where $\psi_n=n^{d+1}\vartheta_{1/n}\star(\vartheta_n\psi)$ and $\langle \cdot ,\cdot \rangle$ denotes the duality pairing of $W^{0,1}_{p'}(\CR^{d+1})$ and $(W^{0,1}_{p'}(\CR^{d+1}))'$.
Since $\psi_n$ converges to $\psi$ in $W^{0,1}_{p'}(\CR^{d+1})$ as $n\to \infty$, we conclude that
$\langle \partial_tu ,\psi_n\rangle \to \langle \partial_tu,\psi \rangle$ as $n\to\infty$. This shows that
$\partial_tu_n\stackrel{*}{\weak}\partial_tu$ in $(W^{0,1}_{p'}(\CR^{d+1}))'$ as $n\to\infty$.

{\em Step 2.} Let us now consider the general case.
We extend $u\in {\mathscr H}^{p,1}(Q(a,b))\cap C_b(\overline Q(a,b))$ to $(3a-2b,2b-a)$, by symmetry, first with respect to $t=b$ and then with respect to $t=a$.
The so obtained function $v$ belongs to ${\mathscr H}^{p,1}(Q(3a-2b,2b-a))\cap C_b(\overline Q(3a-2b,2b-a))$.
Proving that $v\in W^{0,1}_p(Q(3a-2b,2b-a))\cap C_b(\overline Q(3a-2b,2b-a))$ is immediate.
Hence, it remains to prove
that the distributional derivative $\partial_tv$ belongs to
$(W^{0,1}_{p'}(Q(3a-2b,2b-a)))'$. To that end fix $\varphi\in C^{\infty}_c(Q(3a-2b,2b-a))$ and observe that
\begin{align}
\int_{Q(3a-2b,2b-a)}v\partial_t\varphi\,dt\,dx
=\int_{Q(a,b)}u\partial_t\Phi\,dt\,dx\,,
\label{extension-1}
\end{align}
where the function $\Phi=\varphi-\varphi(2b-\cdot,\cdot)-\varphi(2a-\cdot,\cdot)+\varphi(2a-2b+\cdot,\cdot)$
belongs to $W^{0,1}_{p'}(Q(a,b))$. It follows immediately that
$\langle \partial_tv ,\varphi \rangle=\langle \partial_t u ,\Phi \rangle$. The density of $C^{\infty}_c(Q(a,b))$ in $W^{0,1}_{p'}(Q(a,b))$ implies that
$\partial_tv\in (W^{0,1}_{p'}(Q(3a-2b,2b-a)))'$.

We now fix a function $\zeta\in C^{\infty}_c((3a-2b,2b-a))$ such that $\zeta\equiv 1$ in $[a,b]$.
Applying Step 1 to the function $(t,x)\mapsto \zeta(t) v(t,x)$, which belongs to ${\mathscr H}^{p,1}(\CR^{d+1})\cap C_b(\CR^{d+1})$,
we can find a sequence $(u_n)\subset C^{\infty}_c(\CR^{d+1})$ converging
to the function $\zeta v$ locally uniformly in $\CR^{d+1}$ and in $W^{0,1}_p(\CR^{d+1})$, and such that $\partial_tu_n\stackrel{*}{\weak}\partial_t(\zeta v)$ in $(W^{0,1}_{p'}(\CR^{d+1}))'$.
Clearly, $u_n$ converges to $u$ locally uniformly in $\overline Q(a,b)$ and in $W^{0,1}_p(Q(a,b))$. Moreover,
fix $\varphi\in W^{0,1}_{p'}(Q(a,b))$ and denote by $\overline\varphi$ the null extension of $\varphi$ to the whole of $\CR^{d+1}$.
Clearly, $\overline\varphi$ belongs to $W^{0,1}_{p'}(\CR^{d+1})$.
Since
\begin{align*}
\int_{Q(a,b)}\partial_tu_n\varphi\,dt\,dx=
\int_{\CR^{d+1}}\partial_tu_n\overline\varphi\,dt\,dx\,
\end{align*}
and $\partial_t u_n\stackrel{*}{\weak} \partial_t(v\zeta)$ in $(W^{0,1}_{p'}(\CR^{d+1}))'$, from formula \eqref{extension-1}
and since $\zeta'\overline\varphi\equiv 0$ and $\zeta\overline\varphi\equiv\varphi$,
it follows that
\begin{eqnarray*}
 \lim_{n\to\infty}\int_{Q(a,b)}\partial_tu_n\varphi\,dt\,dx &=&
\langle \partial_t(\zeta v),\overline\varphi \rangle
=\int_{Q(a,b)} v\zeta'\overline\varphi\,dt\,dx +\langle \partial_t v,\zeta\overline\varphi\rangle \\
&=& \langle \partial_tv ,\overline\varphi\rangle =\langle \partial_tu ,\varphi\rangle \,.
\end{eqnarray*}
This completes the proof.
\qed
\end{proof}

\begin{lemma}\label{l.bdd}
Let $0\leq a<b <t$ and $x \in \CR^d$. Moreover, assume that $g_0(t, \cdot, x, \cdot) \in L^\infty (Q(a,b))$. Then,
$g(t, \cdot, x, \cdot) \in C_b(\overline Q(a,b))$. Moreover,
if for some $q>1$ we have $\Gamma_1(q,x,a,b)< \infty$
and $\Gamma_2(q,x,a,b)<\infty$,
then  $g(t,\cdot, x, \cdot) \in \mathscr{H}^{p,1}(Q(\tilde a,\tilde b))$ for all $p \in (1,q)$ and any $a<\tilde a<\tilde b<b$.
\end{lemma}

\begin{proof}
By the maximum principle, $g(t,\cdot,x,\cdot) \leq g_0(t,\cdot,x,\cdot)$ almost surely. Hence, $g(t,\cdot,x,\cdot) \in L^\infty (Q(a,b))$. The continuity of the function $g(t,\cdot,x,\cdot)$ follows from Lemma \ref{lem-reg-g}.
To infer that $g(t,\cdot,x,\cdot)$ belongs to $\mathscr{H}^{p,1}(Q(\tilde a,\tilde b))$, for any $\tilde a$ and $\tilde b$ as in the statement of the lemma, we want to use \cite[Lemma 3.2]{mpr10} (see also \cite[Lemma 3.2]{klr13} for the nonautonomous situation). We note that the proof of that lemma remains valid for operators with potential term, provided that both $\Gamma_1(q,x,a,b)<\infty$ and $\Gamma_2(q,x,a,b)< \infty$.
Thus \cite[Lemma 3.2]{klr13} yields
$g \in \mathscr{H}^{p,1}(Q(\tilde a,\tilde b))$ for all $p \in (1, q)$.
\qed
\end{proof}

We next establish the kernel estimates. To that end, we use time-dependent Lyapunov functions. We make the following
assumptions.

\begin{hyp}\label{hyp2}
Fix $0< t\leq 1$, $x \in \CR^d$ and $0< a_0<a<b< b_0<t$.
Let time dependent Lyapunov functions $W_1, W_2$ with $W_1 \leq W_2$ and a weight function
$1 \leq w \in C^{1,2}(Q(0,t))$ be given such that
\begin{enumerate}
\item
the functions $w^{-2}\partial_sw$ and $w^{-2}\nabla_y w$ are bounded on $Q(a_0,b_0)$;
\item
there exist a constant $k>d+2$ and constants $c_1,\ldots,c_7\geq 1$, possibly depending on the interval $(a_0,b_0)$, such that
\[
\begin{array}{ll}
\mathrm{(i)}\quad
w\le c_1w^{\frac{k-2}{k}}W_1^{\frac{2}{k}}\,, & \mathrm{(ii)} \quad |Q\nabla_y w| \leq c_2 w^{\frac{k-1}{k}}W_1^\frac{1}{k}\,,\\
 \mathrm{(iii)}\!\!\! \quad |\mathrm{Tr} (Q D^2 w)| \leq c_3w^{\frac{k-2}{k}}W_1^\frac{2}{k}\,,
& \mathrm{(iv)} \quad |\partial_sw|\leq c_4w^{\frac{k-2}{k}}W_1^\frac{2}{k}\,,\\
 \mathrm{(v)} \quad  |\sum_{i=1}^d D_iq_{ij}| \leq c_5 w^{-\frac{1}{k}}W_2^\frac{1}{k}\,,\quad & \\
\hfill \mathrm{and} &\\
\mathrm{(vi)} \quad |F|\leq c_6w^{-\frac{1}{k}}W_2^{\frac{1}{k}}\,, &\mathrm{(vii)}\quad V^{\frac{1}{2}}\le c_7w^{-\frac{1}{k}}W_2^{\frac{1}{k}},
\end{array}
\]
on $Q(a_0,b_0)$;
\item $g_0(t, \cdot, x, \cdot) \in L^\infty (Q(a_0,b_0))$.
\end{enumerate}
\end{hyp}

Having fixed $t$ and $x$, we write $\rho(s,y) := g(t,s,x,y)$ to simplify notation. We can now prove the main result of this section.

\begin{theorem}\label{t.mainbdd}
Assume Hypotheses \ref{hyp2}. Then there exists a positive constant $C_1$, depending only on $d, k$ and $\eta$,
such that
\begin{align}\label{eq.mainest}
w\rho \leq C_1&\bigg [  c_1^\frac{k}{2}\sup_{s\in (a_0,b_0)}\zeta_{W_1}(s)  + \bigg (
\frac{c_1^{\frac{k}{2}}}{(b_0-b)^{\frac{k}{2}}}+c_2^k
+c_3^{\frac{k}{2}}+c_4^{\frac{k}{2}}\bigg ) \int_{a_0}^{b_0}\zeta_{W_1}(s)\, ds \notag\\
&+ \bigg (c_2^{\frac{k}{2}}c_6^{\frac{k}{2}}+c_5^k+c_6^k+c_7^k\bigg )\int_{a_0}^{b_0} \zeta_{W_2}(s)\, ds \bigg]
\end{align}
in $Q(a,b)$.
\end{theorem}

\begin{proof}
We first assume that the weight function $w$, along with its first order partial derivatives is bounded.
It follows from Hypothesis \ref{hyp2}(2)(i) and (vi) that
\begin{align*}
\Gamma_1(k/2,x,a_0,b_0)^{\frac{k}{2}} & = \int_{Q(a_0,b_0)} |F(s,y)|^{\frac{k}{2}} g(t,s,x,y)\, ds\,dy\\
& \leq \int_{Q(a_0,b_0)} w(s,x)|F(s,x)|^{\frac{k}{2}} g(t,s,x,y)\, ds\,dy\\
& \leq c_6^{\frac{k}{2}}\int_{Q(a_0,b_0)} w(s,y)^{\half}W_2(s,y)^{\half} g(t,s,x,y)\, ds\,dy \\
& \leq c_1^{\frac{k}{4}}c_6^{\frac{k}{2}}\int_{Q(a_0,b_0)} W_2(s,y) g(t,s,x,y)\, ds\,dy < \infty,
\end{align*}
as a consequence of Proposition \ref{p.lyapunov}.
Moreover, using Hypothesis \ref{hyp2}(2)(vii) instead, it follows that
\[
\Gamma_2(k/2,x,a_0,b_0)^{\frac{k}{2}}\leq c_7^k\int_{a_0}^{b_0}\zeta_{W_2}(s,x)\, ds < \infty.
\]
We  thus infer from Lemma \ref{l.bdd} that $g(t,\cdot, x, \cdot) \in L^\infty (Q(a_0,b_0)) \cap \mathscr{H}^{p,1}(Q(a_1,b_1))$ for
all $p \in (1, \frac{k}{2})$, where $a_0<a_1<a<b<b_1<b_0$.
\medskip

Let $\vartheta:\CR\to\CR$ be a smooth function with $\vartheta (s)=1$ for $s \in [a,b]$, $\vartheta (s) = 0$ for $s \geq b_1$, $0\le\vartheta\le 1$ and $|\vartheta'| \leq
2(b_1-b)^{-1}$ in $\CR$. Given $\psi \in C^{1,2}_c(Q(a_1,b_1))$, we put
$\varphi(s,y) := \vartheta(s)^{\frac{k}{2}}w(s,y)\psi (s,y)$. It follows
from \eqref{eq.weak} that
\begin{equation}\label{eq.1}
\int_{Q(a_1,b_1)}\big[\partial_s \varphi (s,y) - \A (s)\varphi (s,y)\big]\rho (s,y)\, ds\, dy = 0.
\end{equation}
We write $\tilde \rho := \vartheta^\frac{k}{2}\rho$ and note that $w\tilde\rho\in \mathscr{H}^{p,1}(Q(a_1,b_1))$ for all $p \in (1, \frac{k}{2})$, since $w$ and its derivatives are bounded. Thus with some standard computations involving integration by parts we derive from
\eqref{eq.1} that
\begin{align*}
&  \phantom{=}\int_{Q(a_1,b_1)}\big [\langle Q\nabla_y (w\tilde{\rho}), \nabla_y \psi\rangle - \psi\partial_s (w\tilde{\rho})\big ] \, ds\,dy\\
& =  \int_{Q(a_1,b_1)} \tilde{\rho} \bigg( 2 \sum_{i,j=1}^d q_{ij}(D_iw)(D_j \psi) - \sum_{i,j=1}^d w (D_iq_{ij})(D_j\psi)+ w \langle F,\nabla_y \psi\rangle\bigg) \, ds\,dy\\
& \qquad-\frac{k}{2}\int_{Q(a_1,b_1)}\rho w\psi \vartheta^{\frac{k-2}{k}}\vartheta'\, ds\,dy\\
& \qquad + \int_{Q(a_1,b_1)}  \psi\big(\tilde{\rho}\mathrm{Tr}(QD^2 w)+\tilde{\rho}\langle F,\nabla_y w\rangle - \tilde{\rho} Vw- \tilde{\rho} \partial_s w\big)\, ds\,dy\,,
\end{align*}
where, with a slight abuse of notation, we denote by $\int_{Q(a_1,b_1)}\psi \partial_s(w\overline\rho)\,ds\,dy$
the pairing between $\partial_s(w\overline\rho)\in (W^{0,1}_{p'}(Q(a_1,b_1)))'$ and $\psi\in W^{0,1}_{p'}(Q(a_1,b_1))$.

We now want to apply \cite[Theorem 3.7]{klr13} to the function $u =w\tilde\rho$ and infer that
there exists a constant $C$, depending only on $\eta, d$ and $k$ (but not on $\|Q\|_\infty)$, such that
\begin{align}\label{eq.inftyest}
\|w\tilde{\rho} \|_{\infty}\leq
C \bigg( &\|w\tilde{\rho}\|_{\infty,2} +
\|\tilde{\rho} Q\nabla_y w\|_{k} +\|\tilde{\rho} Fw\|_{k}
+ \sum_{j=1}^d\bigg\|\tilde{\rho}w\sum_{i=1}^d D_iq_{ij}\bigg\|_{k} + \|\tilde\rho Vw\|_\frac{k}{2}\notag\\
&\!\!\!\!\!\!\!+  \frac{k}{b_1-b}\|\rho w\vartheta^{\frac{k-2}{k}}\|_{\frac{k}{2}}
+ \|\tilde{\rho} \mathrm{Tr}(QD^2w) \|_{\frac{k}{2}}
  + \|\tilde{\rho}\partial_sw \|_{\frac{k}{2}}
+ \|\tilde{\rho}F\cdot\nabla_y w \|_{\frac{k}{2}}  \bigg)\,,
\end{align}
where for $p\in [1,\infty)$ we denote by $\|f\|_p$ the usual $L^p$-norm of the function $f:Q(a_1,b_1)\to\CR$.
Moreover, $\|f\|_{\infty,2}:=\sup_{s\in (a_1,b_1)}\|f(s,\cdot)\|_{L^2(\CR^d)}$.

Note that a major tool in the proof of that theorem is the formula
\begin{equation}
\int_{Q(a_1,b_1)}\vartheta(v-\ell)_+\partial_tv\,dt\,dx
=\frac{1}{2}\left[\int_{\CR^d}\vartheta(v(b_1)-\ell)_+^2\,dx-\int_{\CR^d}\vartheta(v(a_1)-\ell)_+^2\,dx\right]\,.
\label{AA0}
\end{equation}
satisfied by $v=w\tilde\rho$, any $\ell>0$ and any nonnegative function
$\vartheta\in C^{\infty}_c(\CR^d)$, if $p>d+2$.
However, formula \eqref{AA0} is satisfied also in the case $p\le d+2$, which is our situation,
if we additionally assume that $v\in C_b(\overline Q(a_1,b_1))$ (which follows from Lemma \ref{l.bdd}). Its proof can be obtained arguing as
in \cite[Lemma 3.6]{klr13} taking Lemma \ref{lem-cortona-0} into account, with slight and straightforward changes.
Once formula \eqref{AA0} is established, the proof of \eqref{eq.inftyest} follows the same lines
as in \cite[Theorem 3.7]{klr13} with no changes.

We now estimate the terms in the right-hand side of \eqref{eq.inftyest}, using part (2) of Hypothesis \ref{hyp2}. We have
\begin{align*}
\|\tilde\rho Q\nabla_y w\|_k^k & = \int_{Q(a_1,b_1)}|\tilde\rho Q\nabla_y w|^k\, ds\,dy \ \leq c_2^k\int_{Q(a_1,b_1)}\tilde\rho^k w^{k-1}W_1\, ds\, dy \\
 & \leq c_2^k\|\tilde\rho w\|_\infty^{k-1}\int_{a_1}^{b_1}\zeta_{W_1}(s,x)\, ds.
\end{align*}
Let us write $M_k := \int_{a_1}^{b_1}\zeta_{W_k}(s,x)\, ds$ and $\bar{M}:= \sup_{s \in (a_1,b_1)}\zeta_1 (s,x)$. With similar estimates as above, we find
\[
\begin{array}{ll}
\|\tilde \rho F w\|_k \leq c_6\|\tilde\rho w\|_\infty^{\frac{k-1}{k}}M_2^\frac{1}{k}, \qquad & \Big\| \tilde\rho w\sum_{i=1}^d D_iq_{ij} \Big\|_k
\leq c_5\|\tilde\rho w\|_\infty^{\frac{k-1}{k}}M_2^\frac{1}{k},\\[0.5em]
\|\tilde \rho V w\|_\frac{k}{2} \leq c^2_7\|\tilde\rho w\|_\infty^{\frac{k-2}{k}}M_2^\frac{2}{k}, \qquad &
\|\rho  w\vartheta^\frac{k-2}{2}\|_\frac{k}{2}\leq c_1\|\tilde\rho w\|_\infty^{\frac{k-2}{k}}M_1^\frac{2}{k},\\[0.5em]
\|\tilde \rho \mathrm{Tr}(QD^2 w)\|_\frac{k}{2} \leq c_3\|\tilde\rho w\|_\infty^{\frac{k-2}{k}}M_1^\frac{2}{k}, \qquad &
\|\tilde \rho  \partial_s w\|_\frac{k}{2}\leq c_4\|\tilde\rho w\|_\infty^{\frac{k-2}{k}}M_1^\frac{2}{k},\\[0.5em]
\|\tilde \rho F\cdot\nabla_y w\|_\frac{k}{2} \leq \eta^{-1} c_2c_6 \|\tilde\rho w\|_\infty^{\frac{k-2}{k}}M_2^\frac{2}{k},
& \|w\tilde\rho\|_{\infty, 2} \leq c_1^\frac{k}{4} \|w\tilde\rho\|_\infty^\half \bar M^\half .
\end{array}
\]
From \eqref{eq.inftyest} and the above estimates, we obtain the following inequality for $X :=\|w\tilde\rho\|_\infty^\frac{1}{k}$\,:
\begin{eqnarray*}
X^k \leq \alpha X^{\frac{k}{2}}+ \beta X^{k-1} + \gamma X^{k-2}\,,
\end{eqnarray*}
where  $\alpha := Cc_1^\frac{k}{4}\bar M^\half$, $\beta = C\Big (c_2M_1^\frac{1}{k}+(c_6+c_5d)M_2^\frac{1}{k}\Big )$ and
\[
\gamma = C\bigg(\frac{c_1}{b_1-b} +c_3+c_4\bigg)M_1^\frac{2}{k} + C(c_2c_6+c_7^2)M_2^\frac{2}{k}.
\]
Estimating $\alpha X^{k/2}\le \frac{1}{4}{X^k}+\alpha^2$, we find
\begin{equation}
X^k \leq \frac{4}{3}\alpha^2 + \frac{4}{3}\beta X^{k-1}+\frac{4}{3}\gamma X^{k-2}.
\label{estim-X}
\end{equation}
We note that the function
\begin{align*}
f(r)=r^k-\frac{4}{3}\beta r^{k-1}-\frac{4}{3}\gamma r^{k-2}-\frac{4}{3}\alpha^2
=&r^{k-2}\left (r^2-\frac{4}{3}\beta r-\frac{4}{3}\gamma\right )-\frac{4}{3}\alpha^2\\
:=&r^{k-2}g(r)-\frac{4}{3}\alpha^2
\end{align*}
is increasing in $\bigg(\frac{4}{3}\beta+\sqrt{\frac{4}{3}\gamma}+\left (\frac{4}{3}\alpha^2\right )^{\frac{1}{k}},\infty\bigg )$ since
the functions $r\mapsto r^{k-2}$ and $g$ are positive and increasing. Moreover,
\begin{align*}
&f\bigg (\frac{4}{3}\beta+\sqrt{\frac{4}{3}\gamma}+\bigg (\frac{4}{3}\alpha^2\bigg )^{\frac{1}{k}}\bigg )
=\bigg (\frac{4}{3}\beta+\sqrt{\frac{4}{3}\gamma}+\bigg (\frac{4}{3}\alpha^2\bigg )^{\frac{1}{k}}\bigg )^{k-2}\times\\
&\qquad\qquad\quad\quad\times
\bigg [\bigg (\frac{4}{3}\alpha^2\bigg )^{\frac{2}{k}}+\bigg (\frac{4}{3}\bigg )^{\frac{3}{2}}\beta\gamma^{\half}+
2\bigg (\frac{4}{3}\bigg )^{\frac{k+2}{2k}}\alpha^{\frac{2}{k}}\bigg (\frac{\sqrt{3}}{3}\beta+\sqrt{\gamma}\bigg )\bigg ]
-\frac{4}{3}\alpha^2\\
&> \bigg (\frac{4}{3}\alpha^2\bigg )^{\frac{k-2}{k}}\bigg (\frac{4}{3}\alpha^2\bigg )^{\frac{2}{k}}-
\frac{4}{3}\alpha^2=0.
\end{align*}
From these observations and inequality \eqref{estim-X} it follows that
$X \leq \frac{4}{3}\beta+\sqrt{\frac{4}{3}\gamma}+\left (\frac{4}{3}\alpha^2\right )^{\frac{1}{k}}$. Equivalently,
\[
\|\tilde\rho w\|_\infty \le K_1\left (\alpha^2+\beta^k+\gamma^{\frac{k}{2}}\right ),
\]
for some positive constant $K_1$.
Taking into account that $c\geq 1$, one derives \eqref{eq.mainest} from this by plugging in the definitions of
$\alpha, \beta, \gamma$ and, then, letting $a_1\downarrow a_0$ and $b_1\uparrow b_0$.\medskip

To finish the proof of the theorem, it remains to remove the additional assumption on the weight $w$. To that end, we set $w_\eps := \frac{w}{1+\eps w}$. Using Hypothesis \ref{hyp2}(1), we see that $w_\eps$, along with its partial derivatives is bounded.
Straightforward computations show that Part (2) of Hypothesis \ref{hyp2} is satisfied with the same constants $c_1,\ldots,c_7$. Thus the first part of the proof shows that \eqref{eq.inftyest} is satisfied with $w$ replaced with $w_\eps$ and the constants on the right-hand side do not depend on $\eps$. Thus, upon $\eps \downarrow 0$ we obtain
\eqref{eq.inftyest} for the original $w$.
\qed
\end{proof}

\section{The case of general diffusion coefficients}
We now remove the additional boundedness assumption imposed in Section \ref{sect-3}. We do this by approximating general diffusion coefficients with bounded ones, taking advantage of the fact that the constant $C_1$ obtained in Theorem \ref{t.mainbdd} does not depend on the supremum norm of the diffusion coefficients. More precisely, we approximate the diffusion matrix $Q$ as follows.
Given a function $\varphi \in C_c^\infty(\CR)$ such that
$\varphi\equiv 1$ in $(-1,1)$, $\varphi\equiv 0$ in $\CR\setminus (-2,2)$ and $|t\varphi'(t)| \leq 2$ for all $t \in \CR$, we define $\varphi_n(s,x):=
\varphi (W_1(s,x)/n)$ for $s\in [0,t]$ and $x \in \CR^d$. We put
\[
q_{ij}^{(n)}(s,x) :=\varphi_n(s,x)q_{ij}(s,x) + (1-\varphi_n(s,x))\eta\delta_{ij},
\]
where $\delta_{ij}$ is the Kronecker delta, and define the operators $\A_n(s)$ by
\[
\A_n(s) := \sum_{i,j=1}^d q_{ij}^{(n)}(s)D_{ij} + \sum_{j=1}^d F_j(s)D_j - V(s).
\]

We collect some properties of the approximating operators, omitting the easy proof.

\begin{lemma}\label{l.prop}
Each operator $\A_n$ satisfies Hypothesis \ref{hyp1} in $[0,t]$, and its diffusion coefficients are bounded together with their
first-order spatial derivatives. Moreover,
 any time dependent Lyapunov function for the operator $\partial_s -\A (s)$ on $[0,t]$ is a time dependent Lyapunov function for the operator $\partial_s -\A_n(s)$ with respect to the same $h$.
\end{lemma}

It follows that the parabolic equation \eqref{eq.nee} with $\A$ replaced with $\A_n$ is wellposed and the solution is given through an evolution family
$(G_n(r,s))_{0\leq s \leq r \leq t}$. Moreover, for $s<r$ the operator $G_n(r,s)$ is given by a Green kernel $g_n(r,s,\cdot, \cdot)$. We write $\A_n^0:= \A_n + V$ and denote the Green kernel associated to the operators $\A_n^0$ by
$g_n^0$.

We make the following assumptions.

\begin{hyp}\label{hyp3}
Fix $0< t\leq 1$, $x \in \CR^d$ and $0<a_0<a<b< b_0<t$ and assume we are given time dependent Lyapunov functions $W_1, W_2$ with $W_1 \leq W_2 \leq
c_0Z^{1-\sigma}$ for some constants $c_0>0$ and $\sigma \in (0,1)$ and a weight function
$1 \leq w \in C^2(\CR^d)$ such that
\begin{enumerate}
\item
Hypotheses \ref{hyp2}(1)-(2) are satisfied;
\item
$|\Delta_y w| \leq c_8w^\frac{k-2}{k}W_1^\frac{2}{k}$ and $|Q\nabla_y W_1| \leq  c_9w^{-\frac{1}{k}} W_1 W_2^\frac{1}{k}$ on
$[a_0,b_0]\times \CR^d$, for certain constants $c_8, c_9\ge 1$;
\item for $n\in \CN$ we have $g_n^0(t, \cdot, x, \cdot) \in L^\infty(Q(a,b))$.
\end{enumerate}
\end{hyp}

In order to prove kernel estimates for the Green kernel $g$, we apply Theorem \ref{t.mainbdd} to the operators $\A_n$ and then let $n \to \infty$. To do so, we have to show that the operators $\A_n$ satisfy Hypothesis \ref{hyp2}.

\begin{lemma}
\label{lemma-4}
The operator $\A_n$ satisfies Hypothesis \ref{hyp2} with the same constants $c_1$, $c_4$, $c_6$, $c_7$ and with $c_2$, $c_3$ and $c_5$ being replaced,
respectively, by $2c_2$, $c_3+\eta c_8$ and $c_5+4c_9$.
\end{lemma}

\begin{proof}
Since part (1) is obvious and part (3) follows directly from part (3) in Hypothesis \ref{hyp3},
we only need to check part (2) of Hypothesis \ref{hyp2}. Here, the estimates (i), (iv), (vi) and (vii) are obvious, as they do not depend on the diffusion coefficients.
Let us next note that
\[
|\nabla_y w| = |Q^{-1}Q\nabla_y w| \leq \eta^{-1}c_2w^\frac{k-1}{k}W_1^\frac{1}{k},
\]
so  that
\[
|Q_n\nabla_y w| = |\varphi_n Q\nabla_y w + (1-\varphi_n)\eta\nabla_y w|
\leq |Q\nabla_y w| + \eta|\nabla_y w| \leq 2c_2 w^\frac{k-1}{k}W_1^\frac{1}{k}.
\]
This gives (ii) for $Q_n$. As for (iii), we have
\[
|\mathrm{Tr}(Q_nD^2w)| \leq |\mathrm{Tr}(QD^2w)| + \eta |\Delta w| \leq (c_3+\eta c_8) w^{\frac{k-2}{k}}W_1^\frac{2}{k}.
\]
It remains to check (v). We note that
\[
\sum_{i=1}^d D_iq_{ij}^{(n)} = \varphi_n\sum_{i=1}^d D_iq_{ij} + \frac{\varphi'(W_1/n)}{n}\big[ (Q\nabla_y W_1)_j - \eta D_j W_1\big].
\]
As $|t\varphi'(t)| \leq 2$, it follows that
\[
\bigg|\frac{\varphi'(W_1/n)}{n} \big[ (Q\nabla_y W_1)_j - \eta D_jW_1\big ]\bigg| \leq \frac{2}{W_1} (|Q\nabla_y W_1| + \eta |\nabla_y W_1|).
\]
Consequently,
\[
\bigg|\sum_{i=1}^d D_iq_{ij}^{(n)}\bigg| \leq \bigg|\sum_{i=1}^d D_iq_{ij}\bigg|
+ \frac{2}{W_1} (|Q\nabla_y W_1| + \eta |\nabla_y W_1|) \leq (c_5+4c_9) w^{-\frac{1}{k}}W_2^\frac{1}{k}.
\]
This finishes the proof.
\qed
\end{proof}

We shall need the following convergence result for the Green kernels.

\begin{lemma}\label{l.conv}
Fix $r \leq t$ and $x \in \CR^d$ and define $\rho_n(s,y) := g_n(r,s,x,y)$ and $\rho(s,y) := g(r,s,x,y)$ for $s \in [0,r]$ and $y \in \CR^d$. Then
$\rho_n \to \rho$, locally uniformly in $(0,r)\times \CR^d$.
\end{lemma}

\begin{proof}
The proof is obtained as that of \cite[Proposition 2.9]{klr13}. We give a sketch. Using Schauder interior estimates and a diagonal argument,
one shows that for any $f \in C_c^{2+\varsigma}(\CR^d)$ $G_n(\cdot, s)f$ converges to $G(\cdot, s)f$ locally uniformly. This implies
that the measure $\rho_n(s,y)\, dsdy$ converges weakly to the measure $\rho (s,y)\, dsdy$.

On the other hand, \cite[Corollary 3.11]{bkr06} implies that for a compact set $K \subset \CR^d$ and a compact interval $J \subset (0,r)$ we have
$\|\rho_n\|_{C^\gamma (J\times K)} \leq C$ for certain constants $C>0$ and $\gamma \in (0,1)$ independent of $n$. Thus, by compactness,
a subsequence converges locally uniformly to some continuous function $\psi$ which, by the above, has to be $\rho$.
\qed
\end{proof}

We can now state and prove our main result.

\begin{theorem}\label{t.main}
Assume Hypothesis \ref{hyp3}. Then there exists a positive constant $C_1$, depending only on $d, k$ and $\eta$,
such that
\begin{align}\label{eq.mainest2}
w\rho \leq C_1&\bigg [ c_1^\frac{k}{2}\sup_{s\in (a_0,b_0)}\zeta_{W_1}(s)  + \bigg (
\frac{c_1^{\frac{k}{2}}}{(b_0-b)^{\frac{k}{2}}}+c_2^k
+c_3^{\frac{k}{2}}+c_4^{\frac{k}{2}}+c_8^{\frac{k}{2}}\bigg ) \int_{a_0}^{b_0}\zeta_{W_1}(s)\, ds \notag\\
&+ \bigg (c_2^{\frac{k}{2}}c_6^{\frac{k}{2}}+c_5^k+c_6^k+c_7^k+c_9^k\bigg )\int_{a_0}^{b_0} \zeta_{W_2}(s)\, ds \bigg]
\end{align}
in $(a,b)\times\CR^d$.
\end{theorem}

\begin{proof}
We apply Theorem \ref{t.mainbdd} to the operators $\A_n$. Taking Lemma \ref{lemma-4} into account, we obtain
\begin{align}\label{eq.obtainedest}
w\rho_n \leq C_1&\bigg [c_1^\frac{k}{2}\sup_{s\in (a_0,b_0)}\zeta_{1,n}(s)
+ \bigg (c_2^{\frac{k}{2}}c_6^{\frac{k}{2}}+(c_5+4c_9)^k+c_6^k+c_7^k\bigg )\int_{a_0}^{b_0} \zeta_{2,n}(s)\, ds  \notag\\
&+ \bigg (
\frac{c_1^{\frac{k}{2}}}{(b_0-b)^{\frac{k}{2}}}+(2c_2)^k
+(c_3+\eta c_8)^{\frac{k}{2}}+c_4^{\frac{k}{2}}\bigg ) \int_{a_0}^{b_0}\zeta_{1,n}(s)\, ds \bigg],
\end{align}
in $(a,b)$,
where $\zeta_{j,n}(s) := \int_{\CR^d} W_j(x, y) g_n(t,s,x, y)dy$. Note that $\zeta_{j,n}$ is well defined by Proposition \ref{p.lyapunov}, since $W_j$ is also a time dependent Lyapunov function for $\A_n$ by Lemma \ref{l.prop}. Since
$\rho_n \to \rho$ locally uniformly by Lemma \ref{l.conv}, Estimate \eqref{eq.mainest2} follows from
\eqref{eq.obtainedest} upon $n\to \infty$ once we prove that the right-hand sides also converge.

To that end,
it suffices to prove that $\zeta_{j,n}$ converges to $\zeta_{W_j}$ uniformly on $(a_0,b_0)$. Using the estimate
$W_j \leq c_0 Z^{1-\sigma}$ and H\"older's inequality, we find
\begin{align}
|\zeta_{j,n}(s) - \zeta_j(s)| & \leq \int_{\CR^d} W_j(s) |\rho_n(s) - \rho(s)|\, dy\notag\\
& \leq  \int_{B(0,R)} W_j(s) |\rho_n(s) - \rho(s)|\, dy\notag\\
& \qquad + \int_{\CR^d\setminus B(0,R)} W_j(s)\rho_n(s)\, dy + \int_{\CR^d\setminus B(0,R)} W_j(s)\rho (s)\, dy\notag\\
& \leq \|W_j\|_{L^\infty ((a_0,b_0)\times B(0,R))}\|\rho_n-\rho\|_{L^\infty ((a_0,b_0)\times B(0,R))} |B(0,R)|\label{eq.toshow}\\
& \, + c_0 \bigg(\int_{\CR^d\setminus B(0,R)} Z(y) g_n(t,s,x,y)\,dy\bigg)^{1-\sigma}(g_n(t,s,x, \CR^d \setminus B(0,R)))^\sigma\notag\\
& \, + c_0 \bigg(\int_{\CR^d\setminus B(0,R)} Z(y) g(t,s,x,y)\,dy\bigg)^{1-\sigma}(g(t,s,x, \CR^d \setminus B(0,R)))^\sigma,\notag
\end{align}
where $|B(0,R)|$ denotes the Lebesgue measure of the ball $B(0,R)$.
We first note that, as a consequence of Equation \eqref{eq.zest} (which is also valid if $G$ is replaced with $G_n$ since
$Z$ is also a Lyapunov function for $\A_n$), the integrals $\int_{\CR^d} Z(y)g_n(t,s,x,y)\,dy$ are uniformly bounded.
Arguing as in the proof of \eqref{tight}, it is easy to check that the measures $\{g_n(t,s,x,y)\,dy: s\in [0,t]\}$
are tight. Therefore, the last two terms in \eqref{eq.toshow} can be bounded by any given $\eps>0$ if $R$ is chosen large enough. Since $\rho_n \to \rho$ locally uniformly, given $R$, also the first term in \eqref{eq.toshow} can be bounded by $\eps$ if $n$ is large enough. Thus, altogether $\zeta_{j,n} \to \zeta_j$ uniformly on $[a_0,b_0]$. This finishes the proof.
\qed
\end{proof}

\section{Proof of Theorem \ref{t.example}}
Let us come back to the example from Theorem \ref{t.example}. We start by observing that the same computations as in the proof of Proposition
\ref{p.example} show that the function $Z_0(x)=\exp(\delta|x|_{*}^{p+1-m})$ is a Lyapunov function for both the operators $\A_0$ and $\eta\Delta_x-F\cdot\nabla_x$.

To obtain estimates for the Green kernel associated with the operator $\A$, we want to apply Theorem \ref{t.main}. We assume that we are in the situation of Proposition \ref{p.example} and pick $0<\eps_0<\eps_1<\eps_2 < \delta$, where $\delta < 1/\beta$, and $\alpha > \frac{\beta}{m+\beta-2}$. For $\beta\ge 2$, we define the functions
$w, W_1, W_2:[0,t]\times\CR^d$ by
\[
w(s,y) := e^{\eps_0(t-s)^\alpha |y|_{*}^\beta}\quad\mbox{and}\quad W_j(s,y):= e^{\eps_j(t-s)^\alpha |y|_{*}^\beta}.
\]

Let us check the conditions of Theorem \ref{t.main}.
As a consequence of Proposition \ref{p.example}, $W_1$ and $W_2$ are time dependent Lyapunov functions which obviously satisfy $W_1\leq W_2 \leq Z^{1-\sigma}$ for suitable $\sigma$, where $Z(y) := \exp (\delta |y|_{*}^\beta)$. We have to verify that with this choice of $w, W_1$ and $W_2$
Hypothesis \ref{hyp3} is satisfied. As before, we make only computations assuming that $|x| \geq 1$, omitting the details concerning the neighborhood of the origin.

We now fix arbitrary $a_0, b_0\in (0,t)$ with $a_0<b_0$.
Note that $w(s,y)^{-2}\partial_s w(s,y) = -\eps_0\alpha (t-s)^{\alpha-1}|y|^\beta e^{-\eps_0(t-s)^\alpha |y|^\beta}$.
This is clearly bounded. Similarly, one sees that $w^{-2}\nabla_y w$ is bounded.

Let us now turn to part (2) of Hypotheses \ref{hyp2} and \ref{hyp3}. Since $w\leq W_1$, clearly (2)(i) is satisfied with $c_1=1$.
As for (2)(ii), we have
\[
\frac{|Q(s,y)\nabla_y w(s,y)|}{w(s,y)^{1-1/k}W_1(s,y)^{1/k}}  = \eps_0\beta (t-s)^{\alpha} |y|^{\beta -1}(1+|y|^m)e^{-\frac{1}{k}(\eps_1-\eps_0)(t-s)^\alpha |y|^\beta}.
\]
To bound this expression, we note that for $\tau,\gamma, z >0$, we have
\begin{eqnarray*}
z^\gamma e^{-\tau z^\beta} = \tau^{-\frac{\gamma}{\beta}}(\tau z^\beta)^\frac{\gamma}{\beta}e^{-\tau z^\beta} \leq \tau^{-\frac{\gamma}{\beta}}\bigg(
\frac{\gamma}{\beta}\bigg)^\frac{\gamma}{\beta} e^{-\frac{\gamma}{\beta}} =: \tau^{-\frac{\gamma}{\beta}} C(\gamma,\beta)\, ,
\end{eqnarray*}
which follows from the fact that the maximum of the function $t \mapsto t^pe^{-t}$ on $(0,\infty)$ is attained at the point $t=p$. Applying this estimate
in the case where $z = |y|$, $\tau= k^{-1}(\eps_1-\eps_0)(t-s)^\alpha$, $\beta=\beta$ and $\gamma = \beta-1+m$, we get
\begin{align*}
& \phantom{=} \frac{|Q(s,y)\nabla_y w(s,y)|}{w(s,y)^{1-1/k}W_1(s,y)^{1/k}}\\
&  \leq 2\eps_0\beta (t-s)^\alpha \bigg(\frac{\eps_1-\eps_0}{k}\bigg)^{-\frac{\beta
-1+m}{\beta}}(t-s)^{-\alpha \frac{\beta-1+m}{\beta}} C(\beta-1+m, \beta)\\
& =: \bar{c} (t-s)^{-\frac{\alpha(m-1)}{\beta}} \leq \bar{c}(t-b_0)^\frac{-\alpha(m-1)_+}{\beta}\, ,
\end{align*}
for a certain constant $\bar{c}$.\smallskip

Thus we can choose the constant $c_2$ as $\bar{c}(t-b_0)^{-\frac{\alpha(m-1)_+}{\beta}}$, where $\bar{c}$ is a universal constant. Note that $c_2$ depends on the interval $(a_0,b_0)$
only through the factor $(t-b_0)^{-\gamma_2}$. As it turns out, similar estimates show that also for (2)(iii)--(vii) in Hypothesis \ref{hyp2} and in Part (2) of Hypothesis \ref{hyp3} we can choose constants $c_3,\ldots,c_9$ of this form, however with different exponents $\gamma_3,\ldots,\gamma_9$. We now determine the exponents we can choose. To simplify the presentation, we drop constants from our notation and write $\lesssim$ to indicate a constant which merely depends on $d, m, p, r, k, \eps_0, \eps_1, \eps_2$.

As for (iii) we find
\begin{align*}
&\frac{|\mathrm{Tr}(QD^2w(s,y))|}{w(s,y)^{1-2/k}W_1(s,y)^{2/k}}\\
\lesssim &
\big [(t-s)^{\alpha}|y|^{\beta-2+m}+(t-s)^{2\alpha}|y|^{2\beta -2+m}\big ] e^{-\frac{2}{k}(\eps_1-\eps_0)(t-s)^\alpha|y|^\beta}\\
\lesssim & (t-s)^{2\alpha} (t-s)^{-\alpha\frac{2\beta-2+m}{\beta}} \leq (t-b_0)^{-\frac{\alpha (m-2)_+}{\beta}}\,,
\end{align*}
so that here $\gamma_3 = \frac{(m-2)_+}{\beta}$. The estimates
\begin{align*}
\frac{|\partial_sw(s,y)|}{w(s,y)^{1-2/k}W_1(s,y)^{2/k}} & \lesssim (t-s)^{\alpha-1}|y|^{\beta} e^{-\frac{2}{k}(\eps_1-\eps_0)(t-s)^\alpha|y|^\beta}\\
& \lesssim (t-s)^{\alpha -1} (t-s)^{-\alpha} \leq (t-b_0)^{-1}\,,
\end{align*}
\begin{align*}
\frac{w(s,y)^{1/k}|\sum_{i=1}^dD_iq_{ij}(s,y)|}{W_2(s,y)^{1/k}} \lesssim |y|^{m} e^{-\frac{1}{k}(\eps_2-\eps_0)(t-s)^\alpha|y|^\beta}
\lesssim (t-s)^{-\frac{\alpha m}{\beta}} \leq  (t-b_0)^{-\frac{\alpha m}{\beta}}
\end{align*}
and
\begin{align*}
\frac{w(s,y)^{1/k}|F(s,y)|}{W_2(s,y)^{1/k}}=|y|^{p} e^{-\frac{1}{k}(\eps_2-\eps_0)(t-s)^\alpha|y|^\beta}\lesssim (t-s)^{-\frac{\alpha p}{\beta}} \leq  (t-b_0)^{-\frac{\alpha p}{\beta}},
\end{align*}
show that in (iv), resp.\ (v), resp.\ (vi) we can choose $\gamma_4 = 1$, resp.\ $\gamma_5 = \frac{\alpha m}{\beta}$ resp.\ $\gamma_6=\frac{\alpha p}{\beta}$.

A similar estimate as for (vi) shows that in (vii) we can choose $\gamma_7 = \frac{\alpha r}{2\beta}$.\smallskip

Concerning part (2) of Hypothesis \ref{hyp3}, we note that repeating the computations for Hypothesis \ref{hyp2}(2)(ii)-(iii) with $m=0$, we see that in the estimate for
$|\Delta_y w|$ and $|Q\nabla_y W_1|$ we can pick
$c_8=c_9=\bar c$.

Finally for part (3) of Hypothesis \ref{hyp2}, we note that in this special situation the boundedness of the Green kernel for the associated operators without potential term can also be established using time dependent Lyapunov functions. This has been done in \cite{klr13}.\smallskip

We may thus invoke Theorem \ref{t.main}. To that end, given $s \in (0,t)$, we choose $a_0 := \max\{ s -(t-s)/2, s/2\}$ and
$b_0 := s + (t-s)/2$ so that $t-b_0 = (t-s)/2$ and $b_0-a_0 \leq t-s$. Let us note that, as a consequence of  Proposition \ref{p.lyapunov},
\[
\zeta_{W_j}(s,x) \leq \exp \bigg( \int_s^t h(\tau)\, d\tau\bigg) W_j(t,x) = \exp \bigg(\int_s^t h(\tau)\, d\tau \bigg).
\]
Thus, recalling the form of $h$ from the proof of Proposition \ref{p.example}, we see that there exists a constant $H$, depending only on $\alpha, \beta$ and $m$, hence independent of $(a_0,b_0)$, such that
\[
\int_{a_0}^{b_0} \zeta_{W_j}(s)\, ds \leq H (b_0-a_0) \leq H(t-s).
\]
Thus, by Theorem \ref{t.main}, we find that, for a certain constant $C$, we have
\begin{align}
w\rho \leq C \Big( (t-s)^{1-\frac{k}{2}}+(t-s)^{1-\frac{\alpha}{2\beta}((m-1)_++p)k} +(t-s)^{1-\frac{\alpha}{\beta}\Lambda k}
\Big),
\label{eq.lastest}
\end{align}
where $\Lambda=m\vee p\vee\frac{r}{2}$.
To simplify this further, we note first that $$\Lambda \ge \frac{1}{2}((m-1)_++p).$$
Now, let us assume that both
$p> m-1$ and $r > m-2$ so that we can either assume (i) or (ii) in Proposition \ref{p.example}.
Note that in case (i), we have,
by the choice of $\alpha$, that
\[
\frac{\alpha\Lambda}{\beta} \geq \frac{\alpha p}{\beta} > \frac{p}{m+\beta -2} =\frac{p}{p-1} > \frac{1}{2}.
\]
In case (ii), we distinguish the cases $r+m>2$ and $r+m\le 2$. If $r+m>2$ we have
\[
\frac{\alpha\Lambda}{\beta} \geq \frac{\alpha r}{2\beta} > \frac{r}{2(m+\beta -2)} = \frac{r}{r+m -2} > \frac{1}{2},
\]
since $r>m-2$. On the other hand, if $r+m\le 2$, then
\[
\frac{\alpha\Lambda}{\beta} \geq \frac{\alpha p}{\beta} > \frac{p}{p-1} > \frac{1}{2},
\]
Thus, the right-hand side of \eqref{eq.lastest} can be estimated by a constant times
$(t-s)^{1-\frac{\alpha}{\beta}\Lambda k}$.

Therefore, if $p\geq\half (m+r)$, we pick $\beta = p+1-m$. We have,
for $\alpha > \frac{p+1-m}{p-1}$, $\eps < \frac{1}{p+1-m}$,
\begin{align*}
g(t,s,x,y) \leq C(t-s)^{1-\frac{\alpha(m\vee p)k}{p+1-m}}e^{-\eps (t-s)^\alpha |y|_{*}^{p+1-m}},
\end{align*}
for a certain constant $C$. On the other hand, for $p < \half (m+r)$, we pick $\beta = \half (r+2-m)$. So, we obtain
\[
g(t,s,x,y) \leq C (t-s)^{1-\frac{\alpha(2m\vee 2p\vee r)}{2(r+2-m)}k} e^{-\eps (t-s)^\alpha |y|_*^{\half (r+2-m)}},
\]
for $\eps <\frac{2}{r+2-m}$ and $\alpha > \frac{r-m+2}{r+m-2}$ if $r+m>2$, and $\alpha>\frac{r+2-m}{2(p-1)}$ if $r+m\le 2$,
where, again, $C$ is a positive constant independent of $t$ and $s$.
This finishes the proof of Theorem \ref{t.example}.

\end{document}